\documentclass[10pt,a4paper,leqno]{article}
\usepackage{fullpage}
\usepackage{amsfonts,amsmath,amssymb}
\usepackage{amsthm}
\usepackage{graphicx}
\usepackage{subfigure}
\usepackage{sectsty}
\usepackage{authblk}
\allowdisplaybreaks
\theoremstyle{plain}
\newtheorem{theorem}{Theorem}[section]
\newtheorem{lemma}[theorem]{Lemma}

\theoremstyle{definition}

\theoremstyle{remark}

\sectionfont{\large}

\title{Partial Inverse Spectral Problems for Sturm-Liouville Operators with Frozen Arguments on a Star-Shaped Graph}
\date{June 25, 2025}
\author[1]{\small Chung-Tsun Shieh}
\author[2]{\small Tzong-Mo Tsai}
\author[3]{\small Meng-Nien Wu} 
\affil[1]{\footnotesize Department of Applied Mathematics and Data Science, Tamkang University, New Taipei City, 25137, Taiwan\\ctshieh@mail.tku.edu.tw}
\affil[2]{\footnotesize General Education Center, Ming Chi University of Technology, New Taipei City, 24301, Taiwan\\tsaitm@mail.mcut.edu.tw}
\affil[3]{\footnotesize Department of Applied Mathematics and Data Science, Tamkang University, New Taipei City, 25137, Taiwan \\mwu@mail.tku.edu.tw}
\begin{document}
\maketitle
\begin{abstract}
This paper is devoted to the study of a partial inverse spectral problem for Sturm–Liouville operators with frozen arguments on a star-shaped graph. The potentials are assumed to be known a priori on all edges except one, and the objective is to reconstruct the unknown potential on the remaining edge using a subset of the spectral data. A constructive algorithm for solving this problem is presented, which relies on the Riesz basis property of a system of vector functions.
\end{abstract}

\\

\section{Introduction}
In this work, the authors investigate the inverse spectral problem for the nonlocal Sturm-Liouville equation on a star-shaped graph.
Let ${G_p}$ be a star-shaped graph consisting of $p$ equilateral edges $\{e_j\}_{j=1}^p$, one inner vertex $V_0$ and $p$ outer vertices $\{V_j\}_{j=1}^p$, where each $V_j$ is connected to  $V_0$ via  $e_j$. For  convenience, we identify each   $e_j$ with the interval $[0,\pi],$ and $0$ is assigned to the vertex $V_j,\; j=1,2,\dots, p; $  $\pi$ to the inner vertex $V_0.$ Each edge $e_j,\; 1\leq j\leq p-1$ is associated with an ordinary Sturm-Liouville operator, defined as
   \begin{equation}\label{E1}
   	l_{j}(y_{j})(x)=-y_{j}''(x)+q_{j}(x)y_{j}(x), 
   \end{equation}
   or a nonlocal Sturm-Liouville operator with frozen arguments, 
   \begin{equation}\label{E2}
   l_{j}(y_{j})(x)=-y_{j}''(x)+q_{j}(x)\left(\sum_{k=1}^{n_j} y_{j}(a_{k,j})\right), 
   \end{equation}
    and for the last edge $e_{p},$ the differential operator is defined as
   \begin{equation}\label{E3}
   l_{p}(y_{p})(x)=-y_{p}''(x)+q_{p}(x)\left(\sum_{k=1}^{n_p} y_{p}(a_{k,p})\right),
   \end{equation}
   where $F_{j}:=\{a_{k,j}|0<a_{1,j}<...<a_{n_{j},j}<\pi\}$ represents the set of frozen arguments on edge $e_j$  and $q_{j}(x),$ the potential on $e_j,$ is a real-valued $L^2$-function on $(0,\pi)$ for $1\leq j\leq p.$  
   The primary issue addressed is an inverse spectral problem associated with the differential system:
      \begin{equation}\label{E4}
   \begin{pmatrix}
   	l_1(y_1)\\
   	l_2(y_2)\\
   	\vdots\\
   	l_p(y_p)
   \end{pmatrix}(x)=\lambda \begin{pmatrix}
   	y_1\\
   	y_2\\
   	\vdots\\
   	y_p
	\end{pmatrix}(x),
   \end{equation}
   associated with boundary conditions at outer vertices, i.e.,
   \begin{equation}\label{E5}
   y_{1}^{(\alpha_1)}(0)=y_{2}^{(\alpha_2)}(0)=\cdots=y_{p}^{(\alpha_p)}(0)=0,\;\; \alpha_j \in \{0,1\},
   \end{equation}
   and the Kirchhoff's conditions at inner vertex $V_0$, i.e.,
   \begin{equation}\label{E6}
   \; y_1(\pi)=y_2(\pi)=\cdots= y_{p}(\pi)\;  \text{\,and\,} \; \sum_{j=1}^p y'_j(\pi)=0.
   \end{equation}
For brevity, we denote the problem  \eqref{E4}-\eqref{E6} as $L_{G_p}(Q_1,Q_2,\cdots,Q_p),$  where $Q_j$ denotes the pair $(q_j,F_j).$ According to our convention, $l_j$ refers to  the ordinary Sturm-Liouville operator given by \eqref{E1} when  $F_j=\emptyset,$ whereas it represents the nonlocal Sturm-Liouville operator described in  \eqref{E2} otherwise.\\

The nonlocal Sturm-Liouville equation of the form  $l_p(y_p)=\lambda y_p$ on an interval arises in various applications and belongs to the class of so-called loaded equations \cite{Nakh12, app1}.  Such  equations have been widely employed in modeling groundwater dynamics  \cite{Nakh82, NakhBor77}, heating processes \cite{Isk, Nakh76}, and feedback-like phenomena, including  the vibrations of a wire influenced by a magnetic field \cite{Kral, BH21}. The direct and inverse spectral problems associated with $l_p(y_p)=\lambda y_p$ under  boundary conditions $y_{p}^{(\alpha)}(0)=y_{p}^{(\beta)}(\pi)=0$ have been  extensively studied  for  nonlocal Sturm-Liouville operator with a single frozen argument \cite{BBV, BV, BK,Wang20, KuzIrr, Kuz22,TLBCS}.  More recently, results concerning the spectral analysis of nonlocal Sturm-Liouville operators with multiple frozen arguments were established in \cite{ST}.

In this paper, the authors  focus on the partial inverse problem for Sturm-Liouville operators on a star-shaped graph. Specifically, we consider the case where the potential function is known on certain edges and needs to be reconstructed on the remaining ones. Previous studies have investigated inverse spectral problems for ordinary Sturm-Liouville operators on graphs. For instance, in \cite{Yurk05, Yang2010, Bon2018}, it was demonstrated that a  part of the spectrum suffices to determine the potential on a single edge of a star-shaped graph when the potential on the other edges is known. However, the inverse problem for nonlocal Sturm-Liouville operators on graphs remains unexplored. To date, the only known studies in this direction are \cite{Ni12}, which examined an inverse problem for a specific class of nonlocal Sturm-Liouville operators on graphs, and \cite{Hu20}, where the authors derived trace formulas for such operators. Beyond these works, little progress has been made, leaving significant room for further investigation in this area.

The remainder of this paper is organized as follows. In Section 2, we introduce the fundamental concepts and preliminary results necessary for our analysis. In Section 3, we establish a uniqueness theorem for the inverse problem and present a detailed reconstruction algorithm for recovering the unknown potential function.

\section{Preliminaries}
In this section, we present the preliminary results necessary for analyzing the characteristic function of problem \eqref{E4}-\eqref{E6}. To achieve this, we first determine the solutions on each edge. We introduce the notation $\lambda=\rho^2$ and $\tau=\text{Im}\rho.$ For the ordinary Sturm-Liouville equation, the fundamental solutions are given by:
\begin{lemma}[Chapter 1, \cite{FY01}] Denote $S(x,\lambda)$ 	and $C(x,\lambda)$ the solutions of
\begin{equation}\label{E7}
-y''(x)+q(x)y(x)=\lambda y(x),\; 0<x<\pi, 	
\end{equation}
 which satisfy the initial conditions $S(0,\lambda)=C'(0,\lambda)=0$ and $S'(0,\lambda)=C(0,\lambda)=1.$ Then
\begin{align}
S(x,\lambda)&=\dfrac{\sin \rho x}{\rho}+\int_0^xq(t)\dfrac{\sin \rho(x-t)}{\rho }S(t,\lambda)\, dt =\dfrac{\sin\rho x}{\rho}+O(\dfrac{e^{|\tau x|}}{|\rho|^2}), \label{E8}\\
S'(x,\lambda)&=\cos\rho x+\int_0^x q(t)\cos \rho(x-t)S(t,\lambda)\, dt =\cos \rho x+O(\dfrac{e^{|\tau x|}}{|\rho|}), \label{E9}\\
C(x,\lambda)&=\cos\rho x +\int_0^x q(t)\dfrac{\sin \rho(x-t)}{\rho}C(t,\lambda)\, dt =\cos\rho x +O(\dfrac{e^{|\tau x|}}{|\rho|}),\label{E10}\\
C'(x,\lambda)&=-\rho \sin\rho x+ \int_0^x q(t)\cos \rho(x-t)C(t,\lambda)\, dt = -\rho \sin\rho x+O(|\rho|e^{|\tau x|}).\label{E11}
\end{align}
\end{lemma}

For the nonlocal problem, the solution is more complex due to the absence of a uniqueness theorem.
\begin{lemma} Let $0<a_i<\pi,\;i=1,2,\dots,n.$ For a given $\lambda\in \mathbb C\setminus\{0\},$ there exist solutions $S_F(x,\lambda)$ and $C_F(x,\lambda)$ of the equation
\begin{equation}\label{E12}
-y''(x)+q(x)\sum_{k=1}^{n}y(a_k)=\lambda y(x),
\end{equation}
which satisfy the integral equations
\begin{align}
S_F(x,\lambda)&=A(\lambda)\dfrac{\sin \rho x}{\rho }	+\sum_{k=1}^n S_F(a_k,\lambda)\int_0^x q(t)\dfrac{\sin \rho(x-t)}{\rho}\, dt, \label{E13}\\
C_F(x,\lambda)&=B(\lambda)\cos \rho x	+\sum_{k=1}^n C_F(a_k,\lambda)\int_0^x q(t)\dfrac{\sin \rho(x-t)}{\rho}\, dt,\label{E14}
\end{align}
and
\begin{align}
S'_F(x,\lambda)&=A(\lambda)\cos \rho x +\sum_{k=1}^n S_F(a_k,\lambda)\int_0^x q(t)\cos  \rho(x-t)\, dt,\label{E15} \\
C'_F(x,\lambda)&=-B(\lambda)\rho \sin  \rho x +\sum_{k=1}^n C_F(a_k,\lambda)\int_0^x q(t)\cos  \rho(x-t)\, dt,\label{E16}
\end{align}
where $S_F(0,\lambda)=C'_F(0,\lambda)=0,$\; $A(\lambda)$ and $B(\lambda)$ are constants. 
\end{lemma}
\begin{proof}
Note that all solutions of \eqref{E12} satisfying the initial  condition $y(0,\lambda)=0$ must satisfies \eqref{E13} vice versa. By the continuity of $S_F(x,\lambda)$ at $a_k,$  $\{S_F(a_k,\lambda)\}_{i=1}^n$ shall satisfy the linear system
\begin{align}\label{E17}
-A(\lambda)\dfrac{\sin \rho a_k}{\rho}=\left(\int_0^{a_k}q(t)\dfrac{\sin \rho(a_k-t)}{\rho}\, dt -1\right)S_F(a_k,\lambda)+\sum_{j\ne k}S_F(a_j,\lambda)\int_0^{a_k} q(t)\dfrac{\sin \rho(a_k-t)}{\rho}\, dt.
\end{align}
Denote $w_q(x)=\int_{0}^{x}q(t)\frac{\sin \rho(x-t)}{\rho}\,dt,$ then \eqref{E17} turns to
\begin{equation}
Wu=-A(\lambda)v,	
\end{equation}
where
\begin{align}
	W=\begin{bmatrix}
		w_q(a_{1})-1 & w_q(a_1) &  w_q(a_1) &\hdots &w_q(a_1) \\
		w_q(a_2) &w_q(a_2)-1 & w_q(a_2) & \hdots & w_q(a_2) \\
		w_q(a_3) & w_q(a_3) & w_q(a_3)-1 &  \hdots &w_q(a_3) \\
		\vdots &\vdots & \vdots &\ddots & \vdots \\
		w_q(a_n) & w_q(a_n) & w_q(a_n) &\hdots & w_q(a_n)-1
		\end{bmatrix},
		\end{align}
		\begin{align}
u=\begin{bmatrix}S_F(a_1,\lambda)\\ S_F(a_2,\lambda) \\ S_F(a_3,\lambda) \\ \vdots \\S_F(a_n,\lambda)
\end{bmatrix} \text{ and }v=\begin{bmatrix}
\frac{\sin \rho a_1}{\rho} \\
\frac{\sin \rho a_2}{\rho} \\
\frac{\sin \rho a_3}{\rho} \\
\vdots \\ \frac{\sin \rho a_n}{\rho}
\end{bmatrix}.
\end{align}
We can always take $A(\lambda)=1 \text { or } 0$ which depends on whether $v$ is in the column space of $W.$ The existence of $S_F(x,\lambda)$ is thus ensured. Similarly, the existence of $C_F(x,\lambda)$ is ensured. Moreover, equations \eqref{E15} and \eqref{E16} are immediate consequences of \eqref{E13} and \eqref{E14}, respectively.
\end{proof}

In the remainder of this section, we construct the characteristic function $\triangle_{G_p}(\lambda)$ for the problem on the graph. Because the derivation of the characteristic function requires solving a large linear system, we begin by considering a simplified system to illustrate the derivation process. First, we note that if the ordering of the edges is altered, the eigenvalues remain unchanged, and consequently, the characteristic function is invariant up to a constant multiple. For illustration, we consider the less complicated case $p=4,$ where
$$F_1=\{a_{1,1}\},\, F_{2}=\{a_{1,2}, a_{2,2}\},\, F_3=F_4=\emptyset\text{ and }\alpha_j=0 \text{ for } j=1,2,3,4.$$
We denote by $S_{F_j}(x,\lambda)$ the solution of type \eqref{E13} for $l_j(y_j)=\lambda y_j$,\, $j=1,2$ and $S_j(x,\lambda)$ the sine solution of $l_j(y_j)=\lambda y_j$ for $j=3,4.$  For $j=1,2,$ 
\begin{equation}\label{E21}
S_{F_j}(x,\lambda)=A_j(\lambda)\dfrac{\sin \rho x}{\rho }+\sum_{k=1}^{j}S_{F_j}(a_{k,j},\lambda)\int_0^x q_j(x)\dfrac{\sin \rho(x-t)}{\rho}\, dt,
\end{equation}
and for $j=3,4,$
\begin{equation}\label{E22}
S_j(x,\lambda)=\dfrac{\sin \rho x}{\rho}+\int_0^x q_j(t)\dfrac{\sin \rho(x-t)}{\rho}S_j(t,\lambda)\, dt.  	
\end{equation}
$\lambda$ is an eigenvalue of $L_{G_4}(Q_1,Q_2,Q_3,Q_4)$ if and only if there exist $c_1,c_2, c_3, c_4 \in \mathbb C$ so that
\begin{align}
&c_1S_{F_1}(\pi,\lambda)=c_2S_{F_2}(\pi,\lambda)=c_3S_{3}(\pi,\lambda)=c_4S_{4}(\pi,\lambda),\label{E23}\\
& c_1S'_{F_1}(\pi,\lambda)+c_2S'_{F_2}(\pi,\lambda)+c_3S'_{3}(\pi,\lambda)+c_4S'_{4}(\pi,\lambda)=0,\label{E24}
\end{align}
together with the continuity of solutions at frozen arguments, we have
\begin{align}
c_1S_{F_1}(a_{1,1},\lambda)&=c_1A_1(\lambda)\dfrac{\sin \rho a_{1,1}}{\rho}+c_1S_{F_1}(a_{1,1},\lambda)w_{q_1}(a_{1,1})\label{E25},\\
c_2S_{F_2}(a_{1,2},\lambda)&=c_2A_2(\lambda)\dfrac{\sin \rho a_{1,2}}{\rho}+c_2(S_{F_2}(a_{1,2},\lambda)+S_{F_2}(a_{2,2},\lambda))w_{q_2}(a_{1,2})\label{E26},\\
c_2S_{F_2}(a_{2,2},\lambda)&=c_2A_2(\lambda)\dfrac{\sin \rho a_{2,2}}{\rho}+c_2(S_{F_2}(a_{1,2},\lambda)+S_{F_2}(a_{2,2},\lambda))w_{q_2}(a_{2,2})\label{E27}.
\end{align}
Equations  \eqref{E23}-\eqref{E27} yield  the following linear system : 
\begin{equation*}
\begin{bmatrix}
\dfrac{\sin \rho a_{1,1}}{\rho} & w_{q_1}(a_{1,1})-1 & 0 &0 & 0 &0 &0 \\
0 & 0 & \dfrac{\sin \rho a_{1,2}}{\rho} & w_{q_2}(a_{1,2})-1 & w_{q_2}(a_{1,2}) &0 &0 \\
0 & 0 & \dfrac{\sin \rho a_{2,2}}{\rho} & w_{q_2}(a_{2,2}) & w_{q_2}(a_{2,2})-1 &0 &0 \\
\dfrac{\sin \rho \pi}{\rho} &   w_{q_1}(\pi) & -\dfrac{\sin \rho \pi}{\rho} &-w_{q_2}(\pi)& -w_{q_2}(\pi)& 0 & 0 \\
0 & 0 & \dfrac{\sin \rho \pi}{\rho} &w_{q_2}(\pi)& w_{q_2}(\pi)& -S_3(\pi) & 0\\
0 & 0 & 0 & 0 & 0 & S_{3}(\pi) & -S_{4}(\pi) \\
\cos \rho \pi & w'_{q_1}(\pi) &\cos \rho \pi & w'_{q_2}(\pi) &w'_{q_2}(\pi)& S'_3(\pi) & S'_4(\pi)\\
\end{bmatrix}
	\begin{bmatrix}
	c_{1}A_{1}(\lambda) \\ c_{1} S_{F_1}(a_{1,1}) 	\\ c_{2}A_{2}(\lambda) \\ c_{2} S_{F_2}(a_{1,2})\\ c_{2} S_{F_2}(a_{2,2}) \\ c_{3}  \\c_{4}\\
		\end{bmatrix}=\begin{bmatrix}0 \\ 0\\ 0\\ 0 \\ 0\\ 0\\ 0\end{bmatrix},
\end{equation*}
where $w_q(x,\rho)=\int_0^x q(t)\dfrac{\sin \rho(x-t)}{\rho}\, dt $   and for simplicity, we denote $w_q(x)$ as $w_q(x,\rho).$
Hence the characteristic function $\triangle_{G_4}(\lambda)$ of $L_{G_4}(Q_1,Q_2,Q_3,Q_4)$ is

\begin{equation*}
\triangle_{G_4}(\lambda)=\left|\begin{matrix}
\dfrac{\sin \rho a_{1,1}}{\rho} & w_{q_1}(a_{1,1})-1 & 0 & 0 & 0 & 0 & 0 \\
0 & 0 & \sum_{j=1}^{2}\dfrac{\sin \rho a_{j,2}}{\rho} & \sum_{j=1}^{2}w_{q_2}(a_{j,2})-1 & 0 & 0 & 0 \\
0 & 0 & \dfrac{\sin \rho a_{2,2}}{\rho} & w_{q_2}(a_{2,2}) & -1 & 0 & 0 \\
\dfrac{\sin \rho \pi}{\rho} &  w_{q_1}(\pi) & -\dfrac{\sin \rho \pi}{\rho} & -w_{q_2}(\pi)& 0 & 0 & 0 \\
0 & 0 & \dfrac{\sin \rho \pi}{\rho} & w_{q_2}(\pi)& 0 & -S_3(\pi) & 0\\
0 & 0 & 0 & 0 & 0 & S_{3}(\pi) & -S_{4}(\pi) \\
\cos \rho \pi & w'_{q_1}(\pi) & \cos \rho \pi  & w'_{q_2}(\pi) & 0 & S'_3(\pi) & S'_4(\pi)\\
\end{matrix}\right|.\\
\end{equation*}

To expand the determinant $\triangle_{G_4}(\lambda),$ we denote the characteristic function $\triangle^{(\alpha,\beta)}_q(\lambda)$ for the nonlocal Sturm-Liouville operator $L(q,\alpha,\beta)$ which consists in \eqref{E12} associated with boundary condition
\begin{equation}\label{E28}
y^{(\alpha)}(0)=y^{(\beta)}(\pi)=0,\; \alpha,\; \beta\in \{0,1\}.	
\end{equation}
From \cite{ST}, we have
$$\triangle_{q}^{(0,0)}(\lambda)=\begin{vmatrix}
	\dfrac{\sin \rho a_1}{\rho} & w_{q}(a_{1})-1 & w_{q}(a_{1})  &\hdots & w_{q}(a_{1}) \\
	\dfrac{\sin \rho a_2}{\rho} & w_{q}(a_{2})  & w_{q}(a_{2})-1 &\hdots & w_{q}(a_{2})\\
	\vdots & \vdots & \vdots  &\ddots &\vdots \\
	\dfrac{\sin \rho a_n}{\rho} & w_{q}(a_{n}) & w_{q}(a_{n}) &\hdots & w_{q}(a_{n})-1\\
	\dfrac{\sin \rho \pi}{\rho} & w_{q}(\pi)  & w_{q}(\pi)  &\hdots   & w_{q}(\pi)
	\end{vmatrix}=\begin{vmatrix}
	\sum_{j=1}^{n}\dfrac{\sin \rho a_j}{\rho} & \sum_{j=1}^{n} w_{q}(a_{j})-1  \\
	\dfrac{\sin \rho \pi}{\rho} & w_{q}(\pi) 
	\end{vmatrix},$$
and 
$$\triangle_{q}^{(0,1)}(\lambda)=\begin{vmatrix}
	\dfrac{\sin \rho a_1}{\rho} & w_{q}(a_{1})-1 & w_{q}(a_{1})  &\hdots & w_{q}(a_{1}) \\
	\dfrac{\sin \rho a_2}{\rho} & w_{q}(a_{2})  & w_{q}(a_{2})-1 &\hdots & w_{q}(a_{2})\\
	\vdots & \vdots & \vdots  &\ddots &\vdots \\
	\dfrac{\sin \rho a_n}{\rho} & w_{q}(a_{n}) & w_{q}(a_{n}) &\hdots & w_{q}(a_{n})-1\\
	\cos \rho \pi & w'_{q}(\pi)  & w'_{q}(\pi)  &\hdots   & w'_{q}(\pi)
	\end{vmatrix}=\begin{vmatrix}
	\sum_{j=1}^{n}\dfrac{\sin \rho a_j}{\rho} & \sum_{j=1}^{n} w_{q}(a_{j})-1  \\
	\cos \rho \pi & w'_{q}(\pi) 
	\end{vmatrix}.$$

Now, by expanding the determinant with respective to the fifth column and using the multi-linearity of the determinant, we can obtain    

\begin{align*}
\triangle_{G_4}(\lambda)
&=-\left|\begin{matrix}
0 & \sum_{j=1}^{2}\dfrac{\sin \rho a_{j,2}}{\rho} & \sum_{j=1}^{2}w_{q_2}(a_{j,2})-1 & 0 & 0  \\
\triangle_{q_1}^{(0,0)}(\lambda) & -\dfrac{\sin \rho \pi}{\rho} & -w_{q_2}(\pi) & 0 & 0  \\
0 & \dfrac{\sin \rho \pi}{\rho} &  w_{q_2}(\pi) & -S_3(\pi) & 0 \\
0 & 0 & 0 & S_{3}(\pi) & -S_{4}(\pi) \\
\triangle_{q_1}^{(0,1)}(\lambda) & \cos \rho \pi & w'_{q_2}(\pi) &  S'_3(\pi) & S'_4(\pi)\\
\end{matrix}\right|\\
&=\triangle_{q_1}^{(0,0)}(\lambda)\left|\begin{matrix}
\sum_{j=1}^{2}\dfrac{\sin \rho a_{j,2}}{\rho} & \sum_{j=1}^{2}w_{q_2}(a_{j,2})-1 & 0  & 0 \\
\dfrac{\sin \rho \pi}{\rho}& w_{q_2}(\pi) &-S_3(\pi) &0  \\
0 & 0 & S_3(\pi) &-S_4(\pi)\\
\cos \rho \pi  & w'_{q_2}(\pi) & S'_{3}(\pi) & S'_{4}(\pi)\\
\end{matrix}\right|\\
&-\triangle_{q_1}^{(0,1)}(\lambda)\left|\begin{matrix}
\sum_{j=1}^{2}\dfrac{\sin \rho a_{j,2}}{\rho} & \sum_{j=1}^{2}w_{q_2}(a_{j,2})&0 &0 \\
-\dfrac{\sin \rho \pi}{\rho}& -w_{q_2}(\pi) & 0 &0  \\
\dfrac{\sin \rho \pi}{\rho}& w_{q_2}(\pi) & -S_3(\pi) &0\\
0 & 0 &  S_{3}(\pi) & -S_{4}(\pi)
\end{matrix}\right|.
\end{align*}
Moreover, we have
\begin{equation*}
\triangle_{G_4}(\lambda)=\triangle_{q_1}^{(0,0)}(\lambda) \left|\begin{matrix}
\triangle_{q_2}^{(0,0)}(\lambda)  & -S_{3}(\pi) & 0  \\
0 & S_{3}(\pi) & -S_{4}(\pi)  \\
\triangle_{q_2}^{(0,1)}(\lambda) & S'_3(\pi) &  S'_4(\pi) 
\end{matrix}\right|\\
-\triangle_{q_1}^{(0,1)}(\lambda)\left(-S_{4}(\pi)\right)\left(-S_{3}(\pi)\right)\left(-\triangle_{q_2}^{(0,0)}(\lambda)\right).
\end{equation*}
Denote $S_i(\pi,\lambda)=\triangle^{(0,0)}_{q_i}(\lambda),$ $S'_i(\pi,\lambda)=\triangle^{(0,1)}_{q_i}(\lambda)$ for $i=1,2.$ Then
\begin{equation}\label{E30}
	\triangle_{G_4}(\lambda)=\sum_{k=1}^4\triangle^{(0,1)}_{q_k}(\lambda)\prod_{i=1,i\ne k}^{4}\triangle^{(0,0)}_{q_i}(\lambda).
\end{equation}
This form is consistent with the characteristic function of ordinary Sturm-Liouville operator on star-shaped graph. If the graph is of $p$ edges with
$m$ total frozen arguments, the corresponding matrix is of size $p+m$. This example gives useful insights for characteristic function of the boundary value problem $L_{G_p}(Q_1,Q_2,\cdots,Q_p).$  Before proceeding further, we introduce the following notation.
 \begin{equation*} 
\varphi_{\gamma}(x):=\varphi_{\gamma}(x, \lambda)=	\begin{cases}
		\dfrac{\sin \rho x}{\rho} \text{ if  } \gamma=0, \\
		\cos \rho x \text{ if  } \gamma=1,  \\ 
		\end{cases}  \end{equation*}
where $\lambda=\rho^2.$ Next, for a given edge $e_k$, define an $(n_{k}+1)\times(n_{k}+1)$ matrix $M_k$ as
\begin{equation}
M_k=	\begin{bmatrix}
\varphi_{\alpha_k}(a_{1,k})& w_{q_k}(a_{1,k})-1 & w_{q_k}(a_{1,k}) & \hdots   & w_{q_k}(a_{1,k})  \\
\varphi_{\alpha_k}(a_{2,k})& w_{q_k}(a_{2,k}) & w_{q_k}(a_{2,k})-1 & \hdots   & w_{q_k}(a_{2,k})  \\
\vdots & \vdots & \vdots &   \ddots & \vdots \\
\varphi_{\alpha_k}(a_{n_k,k})& w_{q_k}(a_{n_k,k}) & w_{q_k}(a_{n_k,k}) & \hdots   & w_{q_k}(a_{n_k,k})-1  \\
\varphi_{\alpha_k}(\pi) & w_{q_k}(\pi) & w_{q_k}(\pi) & \hdots & w_{q_k}(\pi) 
\end{bmatrix},
\end{equation} 
a corresponding $(n_{k}+2)\times (n_{k}+1)$ matrix $\widehat{M_k}$ as
$$\widehat{M_k}=\begin{bmatrix}
-\varphi_{\alpha_k}(\pi) & -w_{q_k}(\pi) & -w_{q_k}(\pi) & \hdots & -w_{q_k}(\pi) \\
\varphi_{\alpha_k}(a_{1,k})& w_{q_k}(a_{1,k})-1 & w_{q_k}(a_{1,k}) & \hdots   & w_{q_k}(a_{1,k})  \\
\varphi_{\alpha_1}(a_{2,k})& w_{q_k}(a_{2,k}) & w_{q_k}(a_{2,k})-1 & \hdots   & w_{q_k}(a_{2,k})  \\
\vdots & \vdots & \vdots &   \ddots & \vdots \\
\varphi_{\alpha_k}(a_{n_k,k})& w_{q_k}(a_{n_k,k}) & w_{q_k}(a_{n_k,k}) & \hdots   & w_{q_k}(a_{n_k,k})-1  \\
\varphi_{\alpha_k}(\pi) & w_{q_k}(\pi) & w_{q_k}(\pi) & \hdots & w_{q_k}(\pi) 
\end{bmatrix},
$$	and the operation
\begin{equation*}M_k\oplus \widehat{M_j}=
\begin{bmatrix}
\varphi_{\alpha_k}(a_{1,k})& w_{q_k}(a_{1,k})-1 &  \hdots   & w_{q_k}(a_{1,k}) & 0 &0  &\hdots &0 \\
\varphi_{\alpha_1}(a_{2,k})& w_{q_k}(a_{2,k}) &  \hdots   & w_{q_k}(a_{2,k})  & 0 &0  &\hdots & 0\\
\vdots & \vdots & \vdots &   \vdots & \vdots &\vdots   &\ddots &\vdots  \\
\varphi_{\alpha_k}(a_{n_k,k})& w_{q_k}(a_{n_k,k}) &  \hdots   & w_{q_k}(a_{n_k,k})-1  & 0 &0& \hdots &0   \\
\varphi_{\alpha_k}(\pi) & w_{q_k}(\pi) &   \hdots & w_{q_k}(\pi) &-\varphi_{\alpha_j}(\pi) & -w_{q_j}(\pi) &\hdots  &-w_{q_j}(\pi)\\
0 &0 &\hdots&0 &\varphi_{\alpha_j}(a_{1,j})& w_{q_j}(a_{1,j})-1 & \hdots   & w_{q_j}(a_{1,j}) \\
0 &0 &\hdots&0 &\varphi_{\alpha_j}(a_{2,j})& w_{q_j}(a_{2,j}) & \hdots   & w_{q_j}(a_{2,j})\\
\vdots &\vdots &\ddots &\vdots &\vdots &\vdots &\ddots &\vdots \\
0 & 0 & \hdots &0 &\varphi_{\alpha_j}(a_{n_j,j})& w_{q_j}(a_{n_j,j}) & \hdots   & w_{q_j}(a_{n_j,j})-1\\
0 & 0 & \hdots &0  &\varphi_{\alpha_j}(\pi) & w_{q_j}(\pi) &\hdots  &w_{q_j}(\pi)
\end{bmatrix}.
\end{equation*}
The following lemma holds.
\begin{lemma}\label{L2.3} 
Let \( M_1 \oplus \widehat{M}_2 \oplus \widehat{M}_3 \oplus \dots \oplus \widehat{M}_k \) denote the block matrix constructed as described above. Then,
\begin{equation}\label{eqL1}
\det\left( M_1 \oplus \widehat{M}_2 \oplus \widehat{M}_3 \oplus \cdots \oplus \widehat{M}_k \right) 
= \det\left( \mathrm{Diag}(M_1, M_2, M_3, \dots, M_k) \right) 
= \prod_{i=1}^k \triangle^{(\alpha_i,0)}_{q_i}(\lambda),
\end{equation}
where \( \triangle^{(\alpha_i,0)}_{q_i}(\lambda) \) denotes the characteristic function associated with the Sturm--Liouville problem on edge \( e_i \) under the  boundary conditions $y^{(\alpha_i)}_i(0)=y_i(\pi)=0.$ 
\end{lemma}
\begin{proof}
This assertion can be obtained by applying row operations to $$M_1\oplus \widehat{M_2}\oplus \widehat{M_3}\oplus\dots \oplus \widehat{M_k},$$
and we omit the details.
\end{proof}

Note that Equality \eqref{eqL1} remains valid even when $ det(M_j)$ degenerates to a standard determinant expression $\triangle^{(\alpha_j,0)}_{q_j}(\lambda)$, that is, when the differential operator on edge $e_j$ reduces to an ordinary Sturm-Liouville operator on the interval $[0,\pi]$.

\begin{theorem}\label{T2.4}
The characteristic function $\triangle_{G_p}(\lambda)$ of the boundary value problem $L_{G_p}(Q_1,Q_2,\cdots,Q_p)$ is as following:
 \begin{align}
	\triangle_{G_p}(\lambda)&=\sum_{k=1}^p\triangle_{q_k}^{(\alpha_k,1)}(\lambda)\prod_{j=1,j\ne k}^p\triangle_{q_j}^{(\alpha_j,0)}(\lambda),\; \alpha_i\in \{0,1\} \text{ for } 1\leq i\leq p, \label{E30-1}\\
	&=\Delta_{q_p}^{(\alpha_p,0)}(\lambda)\Delta_{G_{p-1}}(\lambda)+\Delta_{q_p}^{(\alpha_p,1)}(\lambda)\prod_{k=1}^{p-1}\Delta_{q_k}^{(\alpha_k,0)}(\lambda).\notag
\end{align}
Here, the boundary condition on edge $e_j$ is $y_j^{(\alpha_j)}(0)=0,$ and $\Delta_{G_{p-1}}(\lambda)$ denote the characteristic function of the subgraph problem $L_{G_{p-1}}(Q_1,Q_2,\dots,Q_{p-1})$ which preserves the boundary conditions. The differential operator on edge $e_p$ can be either an ordinary or a nonlocal Sturm-Liouville operator. 
\end{theorem}

\begin{proof}
The case in which all differential operators on all edges are ordinary has been extensively studied. In this work, we extend this framework by considering cases in which certain edges are characterized by nonlocal Sturm-Liouville operators. For convenience in determining the characteristic function $\triangle_{G_p}(\lambda)$ for $L_{G_p}(Q_1,Q_2,\cdots,Q_p),$ we introduce the following notations:
\begin{equation*}
z_{j}(x, \lambda)=	\begin{cases}
		S_{j}(x,\lambda) \text{ if  } \alpha_{j}=0, \\
		C_{j}(x,\lambda) \text{ if  } \alpha_{j}=1,  \\ 
		\end{cases} \, \,\text{ and }\,\, z_{F_j}(x, \lambda)=	\begin{cases}
		S_{F_j}(x,\lambda) \text{ if  } \alpha_{j}=0, \\
		C_{F_j}(x,\lambda) \text{ if  } \alpha_{j}=1.  \\ 
		\end{cases} 
\end{equation*}
Here, $S_{j}(x,\lambda)$ and $C_{j}(x,\lambda) $ represent the fundamental solutions of equation \eqref{E1}, while  $S_{F_j}(x,\lambda)$ and $C_{F_j}(x,\lambda) $ denote the solutions of equation \eqref{E2}, expressed in the forms given by equations \eqref{E13} and \eqref{E14}, respectively.

We also note that the characteristic function $\triangle^{(\alpha,\beta)}_q(\lambda)$ for the nonlocal Sturm-Liouville operator in (12) associated with boundary condition (28) has the simple representation
\begin{equation}
\triangle_{q}^{(\alpha,\beta)}(\lambda)=\begin{vmatrix}
	\sum _{j=1}^{n}\varphi_{\alpha}(a_{j}) & \sum _{j=1}^{n}w_{q}(a_{j})-1 \\
	\varphi_{\alpha}^{(\beta)}(\pi) & w_{q}^{(\beta)}(\pi) \\
	\end{vmatrix}.
	\end{equation}

We shall prove the assertion by induction on $p$. For $p=2$, there are two cases. The first case consists of one ordinary Sturm-Liouville operator and another operator with frozen arguments, i.e., 
\begin{equation*}
\begin{cases}
-y_{1}''(x)+q_{1}(x)\sum_{k=1}^{n_1}y_{1}(a_{k,1})=\lambda y_{1}(x); \\
-y_{2}''(x)+q_{2}(x)y_{2}(x)=\lambda y_{2}(x); \\
y_{1}^{(\alpha_{1})}(0)=y_{2}^{(\alpha_{2})}(0)=0;\\
y_{1}(\pi)=y_{2}(\pi);\\
y_{1}'(\pi)+y_{2}'(\pi)=0.
\end{cases} 
\end{equation*}
By the continuity  at $a_{k,1}$ and the match conditions, we obtain
\begin{equation*}
\begin{bmatrix}
\varphi_{\alpha_1}(a_{1,1})& w_{q_1}(a_{1,1})-1 & w_{q_1}(a_{1,1}) & \hdots   & w_{q_1}(a_{1,1}) & 0  \\
\varphi_{\alpha_1}(a_{2,1})& w_{q_1}(a_{2,1}) & w_{q_1}(a_{2,1})-1 & \hdots   & w_{q_1}(a_{2,1}) & 0  \\
\vdots & \vdots & \vdots &   \ddots & \vdots &\vdots\\
\varphi_{\alpha_1}(a_{n_1,1})& w_{q_1}(a_{n_1,1}) & w_{q_1}(a_{n_1,1}) & \hdots   & w_{q_1}(a_{n_1,1})-1 & 0  \\
\varphi_{\alpha_1}(\pi) & w_{q_1}(\pi) & w_{q_1}(\pi) & \hdots & w_{q_1}(\pi) & -z_2(\pi)\\
\varphi'_{\alpha_1}(\pi) & w'_{q_1}(\pi) & w'_{q_1}(\pi) & \hdots & w'_{q_1}(\pi) & z'_2(\pi) \\
\end{bmatrix}
	\begin{bmatrix}
	c_{1}D_{1}(\lambda) \\ c_{1}z_{F_1}(a_{1,1},\lambda) \\ c_{1}z_{F_1}(a_{2,1},\lambda) \\ \vdots \\c_{1}z_{F_1}(a_{n_1,1},\lambda)\\ c_2\\
		\end{bmatrix}=\begin{bmatrix}0 \\ 0 \\ 0  \\ \vdots \\0 \\0 \end{bmatrix}.
\end{equation*}
where 
\begin{equation*} 
D_j(\lambda)=	\begin{cases}
		A_j(\lambda) \text{ if  } \alpha_{j}=0, \\
		B_j (\lambda)\text{ if  } \alpha_{j}=1. 
		\end{cases} 
\end{equation*}
Thus, we have the characteristic function 
\begin{equation}\label{E32-1}
\triangle_{G_2}(\lambda)=\begin{vmatrix}
\varphi_{\alpha_1}(a_{1,1})& w_{q_1}(a_{1,1})-1 & 1 & \hdots   & 1 & 0  \\
\varphi_{\alpha_1}(a_{2,1})& w_{q_1}(a_{2,1}) & -1 & \hdots   & 0 & 0  \\
\vdots & \vdots & \vdots &   \ddots & \vdots &\vdots\\
\varphi_{\alpha_1}(a_{n_1,1})& w_{q_1}(a_{n_1,1}) & 0 & \hdots   & -1 & 0  \\
\varphi_{\alpha_1}(\pi) & w_{q_1}(\pi) & 0 & \hdots & 0 & -z_2(\pi)\\
\varphi'_{\alpha_1}(\pi) & w'_{q_1}(\pi) & 0 & \hdots & 0 & z'_2(\pi) \\
\end{vmatrix}=z'_2(\pi)\triangle_{q_1}^{(\alpha_{1},0)}(\lambda)+z_2(\pi)\triangle_{q_1}^{(\alpha_{1},1)}(\lambda).\end{equation}
Moreover, we have
$$
\triangle_{G_2}(\lambda)=\triangle_{q_2}^{(\alpha_{2},1)}(\lambda)\triangle_{q_1}^{(\alpha_{1},0)}(\lambda)+\triangle_{q_2}^{(\alpha_{2},0)}(\lambda)\triangle_{q_1}^{(\alpha_{1},1)}(\lambda).
$$
The second case consists of two Sturm-Liouville operators with frozen arguments, i.e., 
\begin{equation*}
\begin{cases}
-y_{j}''(x)+q_{j}(x)\sum_{k=1}^{n_j}y_{j}(a_{k,j})=\lambda y_{j}(x),\,\,j=1,2; \\
y_{1}^{(\alpha_{1})}(0)=y_{2}^{(\alpha_{2})}(0)=0;\\
y_{1}(\pi)=y_{2}(\pi);\\
y_{1}'(\pi)+y_{2}'(\pi)=0.
\end{cases} 
\end{equation*}
This leas to 
\begin{equation}\triangle_{G_2}(\lambda)=\begin{vmatrix}
\sum_{k=1}^{n_1}\varphi_{\alpha_1}(a_{k,1})& \sum_{k=1}^{n_1}w_{q_1}(a_{k,1})-1 & 0  & 0 &\hdots &0  \\
\varphi_{\alpha_1}(\pi) & w_{q_1}(\pi) &  -\varphi_{\alpha_2}(\pi) & -w_{q_2}(\pi) &  \hdots & -w_{q_2}(\pi)\\
0 &0 & \varphi_{\alpha_2}(a_{1,2})& w_{q_2}(a_{1,2})-1 &  \hdots & w_{q_2}(a_{1,2})\\
0 &0  &\varphi_{\alpha_2}(a_{2,2})& w_{q_2}(a_{2,2}) &  \hdots & w_{q_2}(a_{2,2})\\
\vdots & \vdots  & \vdots &\vdots &  \ddots & \vdots\\
0 &0 & \varphi_{\alpha_2}(a_{n_2,2})& w_{q_2}(a_{n_2,2}) & \hdots & w_{q_2}(a_{n_2,2})-1\\
\varphi'_{\alpha_1}(\pi) & w'_{q_1}(\pi) & \varphi'_{\alpha_2}(\pi) & w'_{q_2}(\pi) &  \hdots & w'_{q_2}(\pi)\\
\end{vmatrix}.
\end{equation}
Thus, we have

\begin{align*}
\triangle_{G_2}(\lambda)&=\begin{vmatrix}
\triangle_{q_1}^{(\alpha_{1},0)}(\lambda) &  -\varphi_{\alpha_2}(\pi) & -w_{q_2}(\pi) & 0 & \hdots & 0\\
0 & \varphi_{\alpha_2}(a_{1,2})& w_{q_2}(a_{1,2})-1 & 1 & \hdots & 1\\
0 & \varphi_{\alpha_2}(a_{2,2})& w_{q_2}(a_{2,2}) & -1 & \hdots & 0\\
\vdots & \vdots  & \vdots &\vdots &   \ddots & \vdots\\
0 & \varphi_{\alpha_2}(a_{n_2,2})& w_{q_2}(a_{n_2,2}) & 0 & \hdots & -1\\
\triangle_{q_1}^{(\alpha_{1},1)}(\lambda) & \varphi'_{\alpha_2}(\pi) & w'_{q_2}(\pi) & 0 & \hdots & 0\\
\end{vmatrix}\\
&=\triangle_{q_1}^{(\alpha_{1},0)}(\lambda)\triangle_{q_2}^{(\alpha_{2},1)}(\lambda)+\triangle_{q_1}^{(\alpha_{1},1)}(\lambda)\triangle_{q_2}^{(\alpha_{2},0)}(\lambda).
\end{align*}

Next, we assume the representations of the characteristic function for the problem with $p-1$ edges are valid. Now, we consider the problem with $p$ edges. 
\begin{equation*}
\begin{cases}
-y_{j}''(x)+q_{j}(x)\sum_{k=1}^{n_j}y_{j}(a_{k,j})=\lambda y_{j}(x),\,j=1,2,...,m;\; 1\leq m\leq p, \\
-y_{j}''(x)+q_{j}(x)y_{j}(x)=\lambda y_{j}(x),\,j=m+1,,...,p;  \\
y_{1}^{(\alpha_{1})}(0)=y_{2}^{(\alpha_{2})}(0)=\cdots=y_{p}^{(\alpha_{p})}(0)=0;\\
y_{1}(\pi)=y_{2}(\pi)=\cdots=y_{p}(\pi);\\
y_{1}'(\pi)+y_{2}'(\pi)+\cdots+y_{p}'(\pi)=0.
\end{cases} 
\end{equation*}

First, we consider the case when $m<p.$ In this setting, the corresponding characteristic function for the problem with 
$p$ edges is given by
\begin{equation}
\begin{vmatrix}&M_1\,\oplus \,\widehat{M_2}\oplus \dots \oplus \widehat{M_m}\oplus \widehat{N_{m+1}}\oplus\dots \oplus \widehat{N_{p}} \\ 
 & \varphi_{\alpha_1}'(\pi)\;  w_{q_1}'(\pi)  \hdots \varphi_{\alpha_m}'(\pi) \;w_{q_m}'(\pi) \; z_{m+1}'(\pi) \hdots  z_{p}'(\pi)
\end{vmatrix},
\end{equation}
where
$$
\widehat{N_{j}}=\begin{bmatrix}
	-z_j(\pi) \\ z_{j}(\pi)
		\end{bmatrix}\; \text{ for }j=m+1,...,p-1 \; \text{ and }\widehat{N_{p}}=\begin{bmatrix}
	-z_p(\pi)
		\end{bmatrix}.
$$
Expanding the determinant with respect to the last column and then applying Lemma 2.3, we obtain
$$\triangle_{G_{p}}(\lambda)=z_p(\lambda)\triangle_{G_{p-1}}(\lambda)+z_p'(\lambda)\prod_{i=1}^{p-1}\triangle^{(\alpha_i,0)}_{q_i}(\lambda).\\
$$
By mathematical induction, \eqref{E30-1} holds for this case and when $p\ge 2.$
Next, consider the situation in which all differential operators are nonlocal. In this case, we add one edge $e_{p+1}$ equipped with an ordinary Sturm-Liouville operator to the graph and denote the new graph  by  $G_{p+1}.$  The previous assertion then yields  
  \begin{align}
  	\triangle_{G_{p+1}}(\lambda)&=z_{p+1}(\lambda)\triangle_{G_{p}}(\lambda)+z_{p+1}'(\lambda)\prod_{i=1}^{p}\triangle^{(\alpha_i,0)}_{q_i}(\lambda),\label{EN1} \\
  	&=\sum_{i=1}^{p+1}\triangle^{(\alpha_i,1)}_{q_i}\prod_{k\ne i}\triangle^{(\alpha_k,1)}_{q_k}. \label{EN2} 
  	\end{align}
  	Noting  that $ z_{p+1}(\lambda)=\triangle_{q_{p+1}}^{(\alpha_{p+1},0)}(\lambda)$ and $ z'_{p+1}(\lambda)=\triangle_{q_{p+1}}^{(\alpha_{p+1},1)}(\lambda).$ Compare \eqref{EN1} and \eqref{EN2},
  	we conclude \eqref{E30-1} holds in  this case as well. This completes the proof.
  	\end{proof}

\section{Main Results}
In the final section, we investigate the following inverse problem and provide a constructive algorithm for the inverse problem.
\vskip 0.2cm
\noindent{\bf Inverse Problem 1.} Given the spectrum $\Lambda $ of $L_{G_{p}}(Q_1, Q_2, Q_3,\dots, Q_p) $, $Q_1,$ $Q_2,$ $\dots,$ $Q_{p-1}$ and $F_p,$ reconstruct $q_p(x).$
\vskip 0.1cm
For simplicity, we only treat the case $\alpha_i=0$ for all $i.$ To conduct a thorough investigation of Inverse Problem 1, it is essential to acquire a deeper understanding of the term $\triangle_{q_p}^{(0, \beta)}$. As referenced in \cite{ST}.
\begin{align}
\triangle^{(0,0)}_{q_p}(\lambda)&=\dfrac{\sin \rho \pi}{\rho}-\dfrac{\sin\rho \pi }{\rho^2} \left(\sum_{k=1}^{n_p}\int_0^{a_{k,p}}q_p(t)\sin\rho(a_{k,p}-t)\, dt  \right)+\sum_{k=1}^{n_p}\dfrac{\sin\rho a_{k,p}}{\rho^2}\int_0^\pi q_p(t)\sin\rho(\pi-t)\, dt, \\
	\triangle^{(0,1)}_{q_p}(\lambda)&=\cos\rho\pi -\dfrac{\cos \rho \pi}{\rho} \left(\sum_{k=1}^{n_p}\int_0^{a_{k,p}}q_p(t) \sin\rho(a_{k,p}-t)\, dt \right)+\sum_{k=1}^{n_p}\dfrac{\sin \rho a_{k,p}}{\rho}\int_0^{\pi}q_p(t)\cos\rho(\pi-t)\, dt.
\end{align}
Note that for $0< a_{k,p}< \pi,$ we have
\begin{align}\label{E34}
	&-\sin \rho\pi \int_0^{a_{k,p}}q_p(t)\sin \rho (a_{k,p}-t)\, dt +\sin\rho a_{k,p} \int_0^\pi q_p(t)\sin \rho(\pi-t)\, dt \\
	=&\dfrac{1}{2}(\int_0^{a_{k,p}}q_p(t)[\cos \rho(\pi+a_{k,p}-t)- \cos \rho(\pi-a_{k,p}+t)]\, dt \notag \\
	&+\int_0^\pi q_p(t)[\cos \rho(\pi -a_{k,p}-t)-\cos \rho (\pi +a_{k,p}-t)]\, dt )\notag\\
	=&\dfrac{1}{2}(  -\int_0^{a_{k,p}}q_p(t)\cos \rho(\pi-a_{k,p}+t)\, dt-\int_{a_{k,p}}^\pi q_p(t)\cos \rho (\pi+a_{k,p}-t)\, dt+\int_0^\pi q_p(t)\cos \rho(\pi-a_{k,p}-t)\, dt)\notag
\end{align}
By substitution of variable, \eqref{E34} turns to
\begin{align*}
&\dfrac{1}{2}(-\int_{\pi-a_{k,p}}^\pi q_p(s-\pi+a_{k,p})\cos\rho s\, ds -\int_{a_{k,p}}^\pi q_p(\pi+a_{k,p}-s)\cos\rho s\, ds  \\ +&\int_{-a_{k,p}}^0 q_p(\pi-a_{k,p}-s)\cos \rho s\, ds +\int_0^{\pi-a_{k,p}} q_p(\pi-a_{k,p}-s)\cos \rho s\, ds ) \\
&=\dfrac{1}{2}(-\int_{\pi-a_{k,p}}^\pi q_p(s-\pi+a_{k,p})\cos\rho s\, ds -\int_{a_{k,p}}^\pi q_p(\pi+a_{k,p}-s)\cos\rho s\, ds   \\
+&\int_0^{a_{k,p}} q_p(\pi-a_{k,p}+s)\cos \rho s\, ds +\int_0^{\pi-a_{k,p}} q_p(\pi-a_{k,p}-s)\cos \rho s\,  ds )\\
&:=\int_0^\pi N_{k,p}(t)\cos \rho t \, dt.
\end{align*}
Hence, we can write
\begin{equation}\label{E37}
	\triangle^{(0,0)}_{q_p}(\lambda)=\dfrac{\sin \rho \pi}{\rho}+\int_0^\pi N_p(t)\dfrac{\cos \rho t}{\rho^2}\, dt,
\end{equation}
where $N_p(t)=\sum_{k=1}^{n_p}N_{k,p}(t)$ and
\begin{equation}\label{E36-1}
\int_0^\pi N_p(t)\, dt =0.	
\end{equation}
 Similarly,
\begin{equation}\label{E38}
\triangle^{(0,1)}_{q_p}(\lambda)=\cos\rho\pi +\int_0^\pi W_p(t)\dfrac{\sin \rho t}{\rho}\, dt.
\end{equation}
To investigate the asymptotic behavior of eigenvalue, we need to study the zeros of the characteristic function. Without lose of generality, we may assume the Sturm-Liouville  operator on $e_i$ is ordinary for $i\leq l$ and nonlocal for $l+1\leq i\leq p.$  When 
$l=0,$  each Sturm–Liouville operator on the corresponding edge is  nonlocal.  According to \cite{Bon2018}, \eqref{E37} and \eqref{E38}, we know that
\begin{equation}\label{E45-1}
\begin{cases}
	\triangle_{q_i}^{(0,0)} (\lambda)&=\dfrac{\sin \rho\pi }{\rho}-\dfrac{\omega_i\cos\rho \pi}{\rho^2}+\dfrac{1}{\rho^2}\int_0^\pi K_i(t)\cos\rho t\,dt,\\
	 \triangle_{q_i}^{(0,1)} (\lambda)&=\cos\rho\pi +\dfrac{\omega_i \sin \rho\pi }{\rho}+\dfrac{1}{\rho}\int_0^\pi H_i(t)\sin \rho t\, dt,
	\end{cases}
	\end{equation}
	for $1\leq i\leq l$  and 
\begin{equation}\label{E45-2}
\begin{cases}
	\triangle_{q_i}^{(0,0)} (\lambda)&=\dfrac{\sin \rho\pi }{\rho}+\dfrac{1}{\rho^2}\int_0^\pi N_i(t)\cos\rho t\,dt,\\
	 \triangle_{q_i}^{(0,1)} (\lambda)&=\cos\rho\pi +\dfrac{1}{\rho}\int_0^\pi W_i(t)\sin \rho t\, dt,
	\end{cases}
	\end{equation}
	otherwise, where $\omega_i=\frac{1}{2}\int_0^\pi q_i(x)\, dx.$ This leads to 
	\begin{align}
 \triangle_{G_p}(\lambda)=&\sum_{1\leq k\leq l}[(\cos(\rho\pi)+\frac{w_k\sin(\rho\pi)}{\rho}+o(1/\rho)) \prod_{1\leq j\leq l, j\ne k}(\dfrac{\sin\rho\pi}{\rho}-\dfrac{w_j\cos(\rho\pi)}{\rho^2}
 +o(1/\rho^2))] (\frac{\sin(\rho\pi)}{\rho}+o(1/\rho^2))^{p-l}\notag \\
 &+(p-l) [\prod_{1\leq j\leq l}(\dfrac{\sin(\rho\pi)}{\rho}-\dfrac{w_j\cos(\rho\pi)}{\rho^2}+o(1/\rho^2))] (\cos(\rho \pi)+o(1/\rho)) (\frac{\sin(\rho\pi)}{\rho}+o(1/\rho^2))^{p-l-1}.\notag 
 \end{align}
  After direct calculation, we have 
 \begin{align}
  \triangle_{G_p}(\lambda)=\dfrac{p\cos(\rho\pi) \sin^{p-1}(\rho \pi)}{\rho^{p-1}}+\dfrac{\sin^{p-2}(\rho\pi)}{\rho^p}[(1-p\cos^2(\rho\pi))\sum_{1\leq k\leq l}\omega_j] +O(1/\rho^{p +1})
 \end{align}
and the spectral sets consists of sequences $\{\lambda_{k,j}\}_{k=1}^\infty$, $j=0,1,2,3,\dots,p-1,$ from the zeros of $\triangle_{G_p}(\lambda)$ in this way:
\begin{align}
\rho_{k,0}&=\sqrt{\lambda_{k,0}}=k-1/2+\dfrac{A_l}{k}+o(1/k),\label{E48-1}\\
\rho_{k,j}&=\sqrt{\lambda_{k,j}}=k+O(1/k),\; j=1,2,3,\dots,p-1,\label{E49-1}
\end{align}
for all $k\in \mathbb N,$ where $A_l=\sum_{1\leq j\leq l} \omega_j.$ Note $A_l=0$ if $l=0.$
 One shall  also notice that there are at most  finite number of eigenvalues are complex. Now, suppose we have two real sequences $\{\mu_{k,j}^2\}_{k=1}^\infty,$ $ j=0,1,$ where 
 \begin{align*}
 \mu_{k,0}&=k+O(1/k), \\
 \mu_{k,1}&=k-\dfrac{1}{2}+\dfrac{A_l}{k}+o(\dfrac{1}{k}).	
 \end{align*}
We know that 
$$\{1\}\cup \{\cos \mu_{k,0} x\}_{k=1}^\infty $$	 and $$\{\sin \mu_{k,j}x\}_{k=1}^\infty$$
form  Riesz bases in $L^2(0,\pi)$  since they are quadratically close to $\{\cos kx\}_{k=0}^\infty $ and $\{\sin(k-1/2)x\}_{k=1}^\infty $ separately.  In the remaining, we shall present a uniqueness theorem for the case $\alpha_j=0,\;  j=1,2,\dots,p. $
\begin{theorem} Assume that $\alpha_j = 0$ for $j = 1, 2, \dots, p$. Let $Q_1, Q_2, \dots, Q_{p-1}$ and $F_p = \{a_{k,p}\}_{k=1}^{n_p}\ne \emptyset$ be given, along with two sequences of non-zero eigenvalues $\{\mu_{k,j}^2\}_{k=1}^\infty,$ $j=0,1,$ where
\begin{align*}
	\mu_{k,0}&=k+O(1/k),\\
	 \mu_{k,1}&=k-1/2+O(1/k),
\end{align*}
where suppose the following conditions hold:
\begin{enumerate}
\item[(i)] Either $\Delta_{q_l}^{(0,0)}(\lambda)\ne 0$  for $l=1,2,\dots, p-1$ and $\triangle_{q_p}^{(0,1)}(\lambda)\ne 0$  for $\lambda=\mu_{k,j}^2$ or  $\triangle_{G_{p-1}}(\lambda)\ne 0 $ and $\triangle_{q_p}^{(0,0)}(\lambda)\ne 0$  for $\lambda=\mu_{k,j}^2$ for $k\in \mathbb N,$ $j=0,1.$
 \item[(ii)] Both the set $\{1\} \cup \{\cos(\mu_{k,0} x)\}_{k \in \mathbb{N}}$ and the set $\{\sin(\mu_{k,1} x)\}_{k \in \mathbb{N}}$ form Riesz bases in $L^2(0, \pi).$
  \item[(iii)] The equation $\sum_{k=1}^{n_p} \sin(\rho a_{k,p}) = 0$ has no solution in $\mathbb{N}.$
 \end{enumerate}
Then the potential function $q_p(x)$ can be uniquely determined.
\end{theorem}
\begin{proof}
The proof can be achieved by following the arguments in \cite{Bon2018} and \cite{ST}.
Denote $Q_p=(q_p,F_p),$
We write
\begin{equation*}
\begin{cases}
\Delta_{q_p}^{(0,0)}(\lambda)&=\dfrac{\sin\rho\pi}{\rho}+\int_0^\pi N_p(x)\dfrac{\cos (\rho x)}{\rho^2}\, dx, \\
\Delta_{q_p}^{(0,1)}(\lambda)&=\cos \rho\pi+\int_0^\pi W_p(x) \dfrac{\sin (\rho x)}{\rho}\, dx. 
\end{cases}
\end{equation*}
Since $\{\mu^2_{k,j}\}_{k=1}^\infty,\; j=0,1$ are eigenvalues of $L_{G_p}(Q_1,Q_2,\dots, Q_p),$
\begin{equation*}
	\Delta_{G_p}(\mu_{k,j}^2)=\sum_{k=1}^p\triangle_{q_k}^{(0,1)}(\mu_{k,j}^2)\prod_{j=1,j\ne k}^p\triangle_{q_j}^{(0,0)}(\mu_{k,j}^2)=0	
\end{equation*}	
By assumption (i), we have
	\begin{equation}\label{Egk}
	\dfrac{\Delta_{q_p}^{(0,1)}(\mu_{k,j}^2)}{\Delta_{q_p}^{(0,0)}(\mu_{k,j}^2)}=-\sum_{l=1}^{p-1}\dfrac{\Delta_{q_l}^{(0,1)}(\mu_{k,j}^2)}{\Delta_{q_l}^{(0,0)}(\mu_{k,j}^2)}:=g_{k,j}.
	\end{equation} 
Hence,
\begin{equation}\label{E40}
	\dfrac{g_{k,j}}{\mu_{k,j}^2}\int_0^\pi N_p(x)\cos( \mu_{k,j} x)\, dx-\dfrac{1}{\mu_{k,j}}\int_0^\pi W_p(x)\sin (\mu_{k,j} x)\, dx =-g_{k,j}\dfrac{\sin (\mu_{k,j}\pi)}{\mu_{k,j}}+\cos(\mu_{k,j} \pi).
		\end{equation}
According to assumption (i), $g_{k,j}\ne0.$ Assumption (ii) and \eqref{E37} give that 
\begin{equation*}
	g_{k,0} =\dfrac{\cos(\mu_{k,0}\pi)+\gamma_{k,0}/k}{\sin(\mu_{k,0}\pi)/\mu_{k,0}+\kappa_{k,0}/k^2}=O(k^2), 
	\end{equation*}
similarly,  	
\begin{equation*}
	 g_{k,1}=\dfrac{\cos(\mu_{k,1}\pi)+\gamma_{k,1}/k}{\sin(\mu_{k,1}\pi)/\mu_{k,1}+\kappa_{k,1}/k^2}=O(1), 
\end{equation*}
where $\{\gamma_{k,j}\}_{k=1}^\infty\in l^2 ,\; \{\kappa_{k,j}\}_{k=1}^\infty \in l^2. $  This leads to 
\begin{align*}
	g_{k,0}\left(\frac{\sin(\mu_{k,0}\pi)}{\mu_{k,0}}+\frac{\kappa_{k,0}}{k^2}\right)=\cos(\mu_{k,0}\pi)+\frac{\gamma_{k,0}}{k},
\end{align*}
henceforth
\begin{equation}
	\dfrac{u_{k,0}^2}{g_{k,0}}\cos\mu_{k,0}\pi-\mu_{k,0}\sin\mu_{k,0}\pi=\dfrac{\kappa_{k,0}\mu_{k,0}^2}{k^2}-\dfrac{\mu_{k,0}^2\gamma_{k,0}}{kg_{k,0}}.
	\end{equation}
This implies $$\{-\mu_{k,0}{\sin \mu_{k,0}\pi}+\dfrac{\mu^2_{k,0}}{g_{k,0}}\cos\mu_{k,0} \pi\}_{k=1}^\infty\in l^2.$$ 
Similarly, 	
$$\{-g_{k,1}\sin \mu_{k,1}\pi+\mu_{k,1}\cos\mu_{k,1} \pi\}_{k=1}^\infty \in l^2.$$
Then \eqref{E40} turns 
	\begin{align}
\begin{cases}
&\int_0^\pi N_p(x)\cos( \mu_{k,0} x)\, dx-\dfrac{\mu_{k,0}}{g_{k,0}}\int_0^\pi W_p(x)\sin (\mu_{k,0} x)\, dx =-\mu_{k,0}{\sin (\mu_{k,0}\pi)}+\dfrac{\mu^2_{k,0}}{g_{k,0}}\cos(\mu_{k,0} \pi), \\
&\dfrac{g_{k,1}}{\mu_{k,1}}\int_0^\pi N_p(x)\cos( \mu_{k,1} x)\, dx-\int_0^\pi W_p(x)\sin (\mu_{k,1} x)\, dx =-g_{k,1}\sin (\mu_{k,1}\pi)+\mu_{k,1}\cos(\mu_{k,1} \pi).
	\end{cases}
\end{align}
Denote 
\begin{equation*}
{V}_{0,0}(x)=\begin{bmatrix}
1\\0 	
\end{bmatrix},\; V_{k,0}(x)= \left(\begin{bmatrix}  \cos \mu_{k,0} x \\ -\dfrac{\mu_{k,0}}{g_{k,0}}\sin \mu_{k,0} x\end{bmatrix}\right),
\text { for } k\in \mathbb N
\end{equation*}
and 
\begin{equation*}
	V_{k,1}=  \left(\begin{bmatrix} -\dfrac{g_{k,1}}{\mu_{k,1}}\cos \mu_{k,1} x\\ \sin \mu_{k,1} x\end{bmatrix}\right), \text{ for } k\in \mathbb N.
	\end{equation*}
By assumption (ii),
	$$ \mathcal{B}:=\{\begin{bmatrix}1 \\ 0\end{bmatrix},\;\; \begin{bmatrix}\cos \mu_{k,0}x  \\ 0\end{bmatrix}_{k\in \mathbb N}, \;\;\begin{bmatrix}0 \\ \sin \mu_{k,1}x \end{bmatrix}_{k\in \mathbb N}\}$$
is a  Riesz base for $L^2(0,\pi)\oplus L^2(0,\pi)$ and  
$$\{V_{k,0}(x)\}_{k=0}^{\infty}\cup \{V_{k,1}(x)\}_{k=1}^{\infty}$$
is quadratically closed to $\mathcal{B},$ hence it forms a Riesz base for $ L^2(0,\pi)\oplus L^2(0,\pi).$ From \eqref{E40}, 
$\begin{bmatrix}N_p(x)\\ W_p(x)\end{bmatrix}$ can be determined uniquely. As $N_p(x)$ is obtained, we have
\begin{align}\label{E42}
\Delta^{(0,0)}_{q_p}(\lambda)&=\dfrac{\sin \rho \pi}{\rho}+\int_0^\pi N_p(x)\dfrac{\cos \rho x}{\rho^2 }\, dx \\
	&=\dfrac{\sin \rho \pi}{\rho}-\dfrac{\sin \rho \pi }{\rho^2}\sum_{k=1}^{n_p}\int_0^{a_{k,p}}q_p(t)\sin \rho(a_{k,p}-t)\, dt +(\dfrac{\sum_{k=1}^{n_p}\sin\rho a_{k,p}}{\rho^2})\int_0^\pi q_p(t)\sin \rho(\pi- t)\, dt. \notag
	\end{align}
By assumption (iii), substitute $\rho$ in \eqref{E42} with $n\in \mathbb N,$  we obtain the Fourier coefficients of $q_p $ in sine series, that allows us to reconstruct $q_p(t).$
\end{proof}

\noindent{\bf Remarks:} \\
1. Note that assumption (ii) is automatically satisfied if the two sequences $\{\mu_{k,j}^2\}_{k=1}^\infty,\; j=1,2, $ of eigenvalues are all real (refer \cite{He2001, RO2004} or  [23, Appendix D, p.163]).\\
2. If the Sturm–Liouville operator on \( e_p \) is ordinary, then one can follow the arguments in \cite{Bon2018} to obtain Weyl's function of the Sturm-Liouville equation on $e_p,$ this leads to  uniqueness and  reconstruction algorithm for this case. We omit the details here.\\

\noindent {\bf Algorithm for reconstruction for $q_p(x)$-nonlocal case.} 
Given the potentials $(Q_1,Q_2,\dots,Q_{p-1})$, the set of frozen arguments $F_p=\{a_{k,p}\}_{k=1}^{n_p}$ on edge $e_p$  and the sequences of eigenvalues  $\{\mu_{k,j}^2\}_{k\in \mathbb N},$  $j=0,1,$ the goal is to determine $q_p(t).$ The computation follows these steps: 
\begin{itemize}
\item[1.] Compute the sequences  $\{g_{k,j}\}_{k\in \mathbb N}$ for $j=0,1 $ by  equation \eqref{Egk}.
\item[2.]Solve equation \eqref{E40} to obtain the vector function:  $$\begin{bmatrix}N_p(x)\\ W_p(x)\end{bmatrix}.$$ 
\item[3.]Compute the Fourier coefficients $c_n$ using equation \eqref{E42}, i.e.,
  $$ c_n=\int_0^\pi q_p(t)\sin n(\pi-t)\, dt=\dfrac{1}{\sum_{k=1}^{n_p}\sin n a_{k,p}} \int_0^\pi N_p(x)\cos n x \, dx. $$
\item[4.] Reconstruct $q_p(t)$ using the Fourier series expansion:
    $$ q_p(t)=\sum_{k=1}^\infty \dfrac{2c_n}{\pi}\sin n(\pi-t) .$$
\end{itemize}

\end{document}